\documentclass{article}
\usepackage{amsmath,amssymb}
\usepackage{amscd}
\usepackage{comment}
\usepackage{color}
\usepackage{enumerate}
\usepackage{cases}
\usepackage{mathtools}
\usepackage[all]{xy}
\usepackage{amsthm}
\usepackage{pb-diagram}
\theoremstyle{definition}
\newtheorem{defi}{Definition}%[section]
\newtheorem{theo}[defi]{Theorem}

\newtheorem{ex}[defi]{Example}
\newtheorem{rem}[defi]{Remark}

\def\SO{{\rm SO}}
\def\SL{{\rm SL}}
\def\U{{\rm U}}

\def\R{{\mathbb R}}

\def\C{{\mathbb C}}

\def\F{{\cal F}}

\def\inum{{\sqrt{-1}}}

\def\vol{{d\mu_{g_M}}}

\def\Lb{{{L}^{2m}_{3,b}}}

\def\Eb{{E_b}}

\begin{document}

\title {Generalized Kazdan-Warner equations on foliated manifolds}
\author {Natsuo Miyatake}
\date{}
\maketitle
\begin{abstract} On compact foliated manifolds, we extend the theorem on the existence and uniqueness of solutions to generalized Kazdan-Warner equations. We provide examples of PDEs that we solve, including the transverse Hitchin equation for a diagonal harmonic metric on basic cyclic Higgs bundles over a 3-dimensional complex codimension one foliated manifold, and its generalizations.
\end{abstract}

\section{Introduction}
Kazdan and Warner \cite{KW1} studied the following elliptic PDE on a Riemannian manifold $(M,g_M)$ which is now called the {\it Kazdan-Warner equation}, in connection with the prescribed Gaussian curvature problem on a compact real surface:
\begin{align}
\Delta_{g_M} f+he^f=c, \label{KW}
\end{align}
where $h$ and $c$ are real functions over $M$, $\Delta_{g_M}\coloneqq d^\ast d$ denotes the geometric Laplacian, and $f$ is a solution of (\ref{KW}). Although the primary motivation in \cite{KW1} to introduce this equation was to solve the prescribed Gaussian curvature problem, the Kazdan-Warner equation itself has been studied in various contexts,  including the relation with the $\U(1)$-gauge theory. In \cite{Miy1}, from the point of view of the moment maps for linear torus actions, a generalization of the Kazdan-Warner equation was introduced. In this paper, we extend \cite[Theorem 1]{Miy1} on compact foliated manifolds. Let $(M,\F)$ be a compact connected foliated manifold with a foliation $\F$. We denote by $T\F\subseteq TM$ the integrable distribution associated with the foliation. A differential $p$-form $\phi$ is said to be {\it basic} if $\phi$ satisfies the following for all $X\in\Gamma(T\F)$:
\begin{align*}
i_X\phi=i_Xd\phi=0,
\end{align*}
where we denote by $i_X$ the interior product. Note that a function $f$ is basic if and only if $Xf=0$ for all $X\in\Gamma(T\F)$. Let $\Omega_B^p(M)$ the space of smooth basic $p$-forms. The space of basic forms is preserved by the exterior derivative: $d:\Omega_B^p(M)\rightarrow \Omega_B^{p+1}(M).$ Let $g_M$ be a Riemannian metric on $M$. We denote by $d_B^\ast:\Omega_B^{p+1}(M)\rightarrow \Omega_B^p(M)$ the $L^2$-adjoint of $d:\Omega^p_B(M)\rightarrow \Omega_B^{p+1}(M)$. We define the {\it basic Laplacian} $\Delta_B:\Omega^p_B(M)\rightarrow \Omega^p_B(M)$ as $\Delta_B\coloneqq d_B^\ast d+dd_B^\ast.$
We consider the following equation on $(M,g_M)$ introduced in \cite[Section1]{Miy1}:
\begin{align}
\Delta_{g_M}\xi+\sum_{j=1}^d a_je^{(\iota^\ast u^j \xi)}\iota^\ast u^j=w. \label{GKW}
\end{align}
We assume that $a_1,\dots, a_d$ and $w$ are all smooth functions. We also suppose that for each $j=1,\dots, d$, if $a_j\neq 0$, then $a_j^{-1}(0)$ is a measure zero set and $\log a_j$ is integrable with respect to the smooth Riemannian metric $g_M$ (see \cite[Section 1]{Miy1}). On a foliated manifold $(M,\F)$, the following holds:
\begin{theo}\label{basic} {\it Suppose that $a_1,\dots, a_d$ and $w$ are all basic with respect to the foliation. Suppose also that the Laplacian preserves the space of basic functions. Then the following are equivalent:
\begin{enumerate}[(i)]
\item \label{basic1} Equation (\ref{GKW}) has a $C^\infty$-solution $\xi$;
\item \label{basic2}
The given functions $a_1, \dots, a_d$ and $w$ satisfy
\begin{align}
\int_M w \ \vol\in\sum_{j\in J_a}\R_{>0}\iota^\ast u^j, \label{w}
\end{align}
where $J_a$ denotes $\{j\in\{1, \dots, d\}\mid \text{$a_j$ is not identically 0}\}$;
\item \label{basic3}Equation (\ref{GKW}) has a basic $C^\infty$-solution $\xi$;
\item \label{basic4}There exists a basic $C^\infty$-function $\xi:M\rightarrow k^\ast$ which is a solution of the following equation:
\begin{align}
\Delta_B\xi+\sum_{j=1}^d a_je^{(\iota^\ast u^j, \xi)}\iota^\ast u^j=w. \label{bGKW}
\end{align}
\end{enumerate}
Moreover if $\xi$ and $\xi^\prime$ are $C^\infty$-solutions of equation (\ref{GKW}), then $\xi-\xi^\prime$ is a constant which lies in the orthogonal complement of $\sum_{j\in J_a}\R\iota^\ast u^j$.
}
\end{theo}

\section{Proof and Example}
Before starting the proof, it is worth noting that what is essentially new in the above theorem compared to the previous \cite[Theorem 1]{Miy1} is the derivation of (\ref{basic3}) and (\ref{basic4}) from the others. Additionally, we emphasize that confirming the equivalence of (\ref{basic3}) and (\ref{basic4}) is easy to check. The remaining portions of the proof follow readily from previously established results, which we shall demonstrate below. To derive (\ref{basic3}), we employ the strategy of restricting the energy functional to the entire space of basic functions and examining its critical points. Then we start the proof.

\begin{proof}[Proof of Theorem \ref{basic}]From \cite[Theorem 1]{Miy1}, we see that (\ref{basic1}) and (\ref{basic2}) are equivalent and that solution of (\ref{GKW}) is unique up to a constant which lies in $(\sum_{j\in J_a}\R\iota^\ast u^j)^\perp$. Clearly, (\ref{basic3}) implies (\ref{basic1}). We also see that (\ref{basic3}) and (\ref{basic4}) are equivalent since Laplacian $\Delta_{g_M}$ preserves $\Omega_B^0(M)$ if and only if the following holds for all $f\in \Omega^0_B(M)$:
\begin{align*}
\Delta_{g_M}f=\Delta_Bf.
\end{align*}
Therefore it is enough to show that (\ref{basic2}) implies (\ref{basic3}). As we remarked above, we employ the strategy of restricting the energy functional $E$ introduced in \cite{Miy1} to the entire space of basic functions and examining its critical points. Let $L^{2m}_3(M,k^\ast)$ be the subspace of $k^\ast$-valued $L^{2m}_3$-functions. Here, we denote by $L^p_k$ the Sobolev space that contains $L^p$-functions whose weak derivatives up to order $k$ have finite $L^p$-norms. We define a subspace $\Lb(M,k^\ast)$ of $L^{2m}_3(M,k^\ast)$ as follows:
\begin{align*}
\Lb(M,k^\ast)\coloneqq \{\xi\in L^{2m}_3(M,k^\ast)\mid \text{$\xi$ is a $k^\ast$-valued basic function.}\}.
\end{align*}
Let $E$ be the energy functional defined in \cite[Definition 1]{Miy1}. We denote by $\Eb$ the energy functional $E$ restricted to $\Lb$:
\begin{align*}
\Eb\coloneqq \left.E\right|_{\Lb(M,k^\ast)}.
\end{align*}
Then we see that $\Eb$ has a critical point if and only if there exists a smooth basic solution $\xi$ of (\ref{GKW}), following the same argument as in the proof of \cite[Lemma 1]{Miy1}. Note that we have used the assumption that $a_1, \dots, a_d$ and $w$ are all basic, and that $\Delta_{g_M}$ preserves $\Omega^0_B(M)$. Therefore, the problem reduces to finding a critical point of $\Eb$. We can show that $\Eb$ has a critical point under the assumption of (\ref{basic2}) by using the same argument as in the proof of \cite[Theorem 1]{Miy1}, taking care that the Laplacian $\Delta_{g_M}$ preserves the space of basic functions, and proving the same lemma as \cite[Lemma 4]{Miy1} for $E_b$. Then we see that (\ref{basic2}) implies (\ref{basic3}).
%Then we see that $\Eb$ has a critical point if and only if there exists a smooth basic solution $\xi$ of (\ref{GKW}) from the same argument as in the proof of \cite[Lemma 1]{Miy1}. Note that we have used here the assumption that $a_1,\dots, a_d$ and $w$ are all basic and $\Delta_{g_M}$ preserves $\Omega^0_B(M)$. Therefore the problem reduces to find a critical point of $\Eb$. We can show that $\Eb$ has a critical point under the assumption of (\ref{basic2}) by using the same argument as in the proof of \cite[Theorem 1]{Miy1}, if we take care that the Laplacian $\Delta_{g_M}$ preserves the space of basic functions in proving the same lemma as \cite[Lemma 4]{Miy1} for $E_b$. Then we see that (\ref{basic2}) implies (\ref{basic3}).
\end{proof}
\begin{rem} If $(M, g_M)$ is a Riemannian foliation, then Laplacian $\Delta_{g_M}$ preserves $\Omega_B^0(M)$.
\end{rem}
\begin{rem} Our work is motivated by the recent progress in the study of the Kobayashi-Hitchin correspondence and the gauge theory on foliated manifolds \cite{BK1, BK2, BH1, KLW1, WZ1}.
\end{rem}
\begin{rem} As a corollary of Theorem \ref{basic}, the same claim as \cite[Theorem 1]{Miy3} holds for basic Higgs bundles with basic Higgs fields, which may not necessarily be holomorphic (see \cite{BH1, BK1, BK2, WZ1} for fundamental facts about basic vector bundles and basic Higgs bundles). Note that basic vector bundles are also called foliated vector bundles or transverse vector bundles in \cite{BH1,WZ1}, and that basic Higgs bundles are called transverse Higgs bundles in \cite{WZ1}. 
\end{rem}

As indicated in \cite{Miy1, Miy4}, examples of equation (\ref{GKW}) include the Hitchin equation for a diagonal harmonic metric on cyclic Higgs bundles and its generalizations introduced in \cite{Miy4} (for cyclic Higgs bundles, see \cite{ALS1, Bar1, DL1, Miy1, Miy2, Miy4} and the references therein). In addition, we demonstrate that equation (\ref{GKW}) also encompasses the transverse Hitchin equation for a diagonal harmonic metric on basic cyclic Higgs bundles over 3-dimensional complex codimension one foliated manifolds and its generalizations:

\begin{ex}\label{example} Let $X$ be a compact connected orbifold Riemann surface. We take a family of orbifold charts $(U_i, \varphi_i:\tilde{U}_i\rightarrow U_i, \Gamma_i)_{i\in I}$ such that $X=\bigcup_{i\in I}U_i$ and that the gluing condition is satisfied (see \cite[Chapter 4]{BG1}). Let $\pi:M\rightarrow X$ be a surjective submersion from a compact connected smooth real 3-dimensional manifold $M$. Note that if we take a Riemannian metric of $X$, then the $\SO(2)$-frame bundle of $X$ is a compact connected smooth manifold (see \cite[Chapter4]{BG1}) and the projection is a surjective submersion. The submersion $\pi$ defines a natural foliation structure on $M$, which is denoted by $\F$. Let $g_X$ be a K\"ahler metric on $X$ and $\omega_X$ the corresponding K\"ahler form (see \cite{BG1, BK2} for the definition of a K\"ahler metric on a complex orbifold). For each $p\in M$, there exists an open neighborhood $V_p$ of $p$ such that $\pi(V_p)\subseteq U_i$ for an $i\in I$ and that there exists a smooth map $\tilde{\pi}:V_p\rightarrow \tilde{U}_i$ satisfying $\pi\left.\right|_{V_p}=\varphi_i\circ \tilde{\pi}$. We define a bilinear form $g_T\in\Gamma(T^\ast M\otimes T^\ast M)$ as $(g_T)_p(u,v)\coloneqq g_X(d\tilde{\pi}(u),d\tilde{\pi}(v))$ for each $p\in M$ and $u,v\in T_p M$. Similarly, we have a transverse K\"ahler form $\omega_T\in\Omega_B^2(M)$ from the K\"ahler form $\omega_X$. We denote by $\Lambda:\Omega_B^{p,q}(M)\rightarrow \Omega_B^{p-1,q-1}(M)$ the dual of $\omega_T\wedge$. We take a Riemannian metric $g_M$ on $M$ such that $g_M(s,t)=g_T(s,t)$ for all $s,t\in \Gamma(T\F^\perp)$. Let $K_i\rightarrow \tilde{U}_i$ be the canonical bundle of $\tilde{U}_i$ for each $i\in I$. For each $p\in M$, we denote by $K_{V_p}$ the pullback of $K_i$ by $\tilde{\pi}$. Then by patching together a family of complex line bundles $(K_{V_p}\rightarrow V_p)_{p\in M}$, we have a basic holomorphic line bundle over $M$, which is denoted by $K\rightarrow M$. 
The basic holomorphic line bundle $K$ is naturally regarded as a subbundle of $T^\ast M\otimes_\R \C$. Suppose that the basic first Chern class of $K$ is negative. We choose a basic holomorphic line bundle $K^{\frac{1}{2}}$ such that $K^{\frac{1}{2}}\otimes K^{\frac{1}{2}}\simeq K$. Note that in general $K^{\frac{1}{2}}$ can not be pushed forward to a holomorphic line orbibundle over $X$ (see \cite[Example 4.9 and Theorem 4.17]{BK2}). We define a basic holomorphic vector bundle $E$ as $E\coloneqq K^{\frac{r-1}{2}}\oplus K^{\frac{r-3}{2}}\oplus\cdots\oplus K^{-\frac{r-3}{2}}\oplus K^{-\frac{r-1}{2}}$. We take a basic holomorphic section $q$ of $K^r\rightarrow X$ and we define a basic Higgs field $\Phi(q)$ as 
\begin{align*}
\Phi(q)\coloneqq 
\left(
\begin{array}{cccc}
0 & && q\\
1 & \ddots && \\
&\ddots&\ddots& \\
&&1&0
\end{array}
\right),
\end{align*}
where $1$ is considered to be a section of $K\otimes K^{-1}$. We take a basic Hermitian metric $h$ on $E$ such that $h$ splits as $h=(h_1,\dots, h_r)$ with respect to the above decomposition of $E$ and that $h_j=h_{r+1-j}^{-1}$ for all $j=1,\dots, r$. From the second assumption, the following $S:E\rightarrow E^\ast$ is isometric with respect to the metric $h$:
\begin{align*}
S\coloneqq \left(
\begin{array}{ccc}
& &1\\
&\reflectbox{$\ddots$} &\\
1&&
\end{array}
\right).
\end{align*}
Let $f_1,\dots, f_r:M\rightarrow \R$ be basic functions satisfying $f_j=-f_{r+1-j} \ (j=1,\dots, r)$ and $f_1+\cdots +f_r=0$. Let $V$ be a real vector space defined as $V\coloneqq \{(x_1,\dots, x_r)\in \R^r\mid x_j=-x_{r+1-j}, x_1+\cdots +x_r=0\}$. We also define $v_j\in\R^r$ $(j=1,\dots, r)$ as
\begin{align*}
&v_j\coloneqq u_{j+1}-u_j, \ \text{for $j=1,\dots, r-1$}, \\
&v_r\coloneqq u_1-u_r,
\end{align*}
where we denote by $u_1,\dots, u_r$ the canonical basis of $\R^r$. Let $k_1,\dots, k_r$ be non-negative functions defined as 
$k_j\coloneqq |1|_{h,g_T}^2 \ (j=1,\dots, r-1)$ and 
$k_r\coloneqq |q|_{h,g_T}^2$,
where the norm is determined by the Hermitian metric $h$ and the transverse K\"ahler metric $g_T$. Then the transverse Hitchin equation \cite{BK1, WZ1} for a basic Hermitian metric $(e^{f_1}h_1,\dots, e^{f_r}h_r)$ is the following:
\begin{align}
\Delta_B\xi+\sum_{j=1}^r4k_je^{(v_j,\xi)}v_j=-2\inum\Lambda F_h,\label{basic Hitchin eq}
\end{align}
where $\xi$ is a $V$-valued function defined as $\xi\coloneqq (f_1,\dots, f_r)$, and we denote by $F_h$ the curvature of the Chern connection of the metric $h$. 
From the symmetry of the Higgs field $\Phi(q)$ for the isomorphism $S$, equation (\ref{basic Hitchin eq}) is well-defined as a PDE for a $V$-valued function $\xi$. Equation (\ref{basic Hitchin eq}) is a special case of equation (\ref{bGKW}) and one can check that in this case condition (\ref{w}) is satisfied (see also \cite[Remark 2 and Section 2]{Miy1}). Therefore equation (\ref{basic Hitchin eq}) has a unique $V$-valued solution $\xi$. From the solution of (\ref{basic Hitchin eq}), we obtain an $\SL(r,\R)$-harmonic bundle over $M$. Moreover, the generalizations of the Hitchin equation for cyclic Higgs bundles presented in \cite{Miy4} can also be applied to the equation (\ref{basic Hitchin eq}).
\end{ex}
\begin{rem} If one of $a_1,\dots, a_d,w$ is not basic, then a solution of equation (\ref{GKW}) (especially, the solution of \cite[equation (1)]{Miy4}) is not necessarily basic. The author does not know in this case how the solution of equation \cite[equation (1)]{Miy4} is related to harmonic bundles.
\end{rem}
\begin{rem} It is not necessary to restrict ourselves to 3-dimensional manifolds in Example \ref{example}. The construction in Example \ref{example} is possible for arbitrary complex codimension one foliated manifolds. If $K^{1/2}$ is a pull-back of an orbibundle of $X$, then from the solution of equation (\ref{basic Hitchin eq}) we obtain a solution of the Hitchin equation for a diagonal harmonic metric on a cyclic Higgs bundle over an orbifold Riemann surface $X$, which was introduced in \cite{ALS1}. We emphasize that a basic cyclic Higgs bundle over $M$ is not necessarily a pull-back of a cyclic Higgs bundle over $X$. We refer the reader to \cite{BK2}, which discusses the correspondence between the harmonic bundles over compact quasi-regular Sasakian manifolds and the harmonic bundles over compact K\"ahler orbifolds obtained as quotients by the circle action on quasi-regular Sasakian manifolds.
\end{rem}
\begin{rem} The above construction of diagonal harmonic metrics on basic cyclic Higgs bundles is easily generalized to basic $G$-cyclic Higgs bundles, where $G$ is a complex connected simple Lie group (see \cite[Section 2]{Miy1}, \cite{Miy2}).
\end{rem}
\begin{rem} As with usual cyclic Higgs bundles, the solution of equation (\ref{basic Hitchin eq}) can also be constructed by using the result of the Kobayashi-Hitchin correspondence for basic Higgs bundles \cite{BK1, WZ1} and the invariance of the harmonic metric under cyclic group actions (see \cite{Bar1, DL1}).
\end{rem}

\medskip
\noindent
{\bf Acknowledgements.} I would like to express my gratitude to Hisashi Kasuya for valuable discussions, and for explaining his work with I. Biswas to me.

\noindent
E-mail address 1: natsuo.miyatake.e8@tohoku.ac.jp

\noindent
E-mail address 2: natsuo.m.math@gmail.com \\

\noindent
Mathematical Science Center for Co-creative Society, Tohoku University, 468-1 Aramaki Azaaoba, Aoba-ku, Sendai 980-0845, Japan.

\end{document}